\documentclass{scrartcl}
\usepackage[margin=1.1in]{geometry}
\title{Digitally delicate primes}
\author{Jackson Hopper\thanks{\emph{email:} \texttt{jacksonh@uga.edu}} \qquad
Paul Pollack\thanks{\emph{email:} \texttt{pollack@uga.edu}}\\ Department of Mathematics\\ University of Georgia\\ Athens, Georgia 30602}
\date{\vspace{-3.2ex}}

\usepackage[utf8]{inputenc}
\usepackage[english]{babel}

\usepackage{amsmath}
\usepackage{amsthm}
\usepackage{amssymb}
\usepackage{enumerate}
\usepackage{microtype}
\usepackage{url}

\newcommand{\Ss}{\mathcal{S}}
\newcommand{\Pp}{\mathcal{P}}

\newtheorem{theorem}{Theorem}[section]
\newtheorem{corollary}[theorem]{Corollary}
\newtheorem{lemma}[theorem]{Lemma}

\begin{document}
\maketitle

\begin{abstract} \noindent Tao has shown that in any fixed base, a positive proportion of prime numbers cannot have any digit changed and remain prime. In other words, most primes are ``digitally delicate''. We strengthen this result in a manner suggested by Tao: A positive proportion of primes become composite under any change of a single digit and any insertion a fixed number of arbitrary digits at the beginning or end.
\end{abstract}

%
%

\section{Introduction}
In a short note published in 2008, Tao \cite{tao11} proved the following theorem:
\begin{theorem}
\label{Tao}
Let $K \ge 2$ be an integer. For all sufficiently large integers $N$, the number of primes $p$ between $N$ and $(1 + 1/K)N$ such that $|kp + ja^i|$ is composite for all integers $1 \le a, |j|, k \le K$ and $0 \le i \le K \log{N}$ is at least $c_K \frac{N}{\log{N}}$ for some constant $c_K > 0$ depending on only $K$.
\end{theorem}

The following consequence is immediate, in view of the prime number theorem (or Chebyshev's weaker estimates).
\begin{corollary}\label{cor:tao} Fix a base $a\ge 2$. A positive proportion of prime numbers become composite if any single digit in their base $a$ expansion is altered.
\end{corollary}

The infinitude of the primes appearing in Corollary \ref{cor:tao} had earlier been shown by Erd\H{o}s \cite{erdos79}. (He assumes $a=10$ but the argument generalizes in an obvious way.) When $a=10$, these ``digitally delicate'' primes are tabulated as sequence \texttt{A050249} in the OEIS, where they are called ``weakly prime''.


At the conclusion of \cite{tao11}, Tao suggests a few ways his result could possibly be improved. In this paper we establish one of the suggested generalizations:
\begin{theorem}\label{thm:taoimproved}
Fix an integer $K \ge 2$. There is a constant $c_K > 0$ such that the following holds for all sufficiently large $N$: Let $\Ss_N \subseteq {[{-KN}, KN]}$ be an arbitrary set of integers of cardinality at most $K$. Let $K_N$ be the number of primes $N \leq p \leq (1 + 1/K)N$ such that $|kp + ja^i + s|$ is either equal to $p$ or composite for all combinations of integers $a, i, j, k$, and $s$ where $1 \leq a, |j|, k \leq K$, $0 \leq i \leq K \log{N}$, and $s \in \Ss_N$. Then $K_N \ge c_K \frac{N}{\log{N}}$.
\end{theorem}

This immediately yields the following strengthening of Corollary \ref{cor:tao}.

\begin{corollary} In any fixed base, a positive proportion of prime numbers become composite if one modifies any single digit and appends a bounded number of digits at the beginning or end.
\end{corollary}



As in Tao's work, the key idea of the proof is to use a partial covering along with an upper bound sieve. The following well-known estimate plays a critical role (see \cite[Theorem 2.2, p 68]{HR74}, \cite[Corollary A.2]{tao11}).
\begin{lemma}[Brun/Selberg upper bound]
\label{Th:selberg}
Let $W$ and $b$ be positive integers and let $k$ and $h$ be non-zero integers. If $x$ is sufficiently large (depending on $W$ and $b$), the number of primes $m \le x$ where $m \equiv b \pmod{W}$ and $|km + h|$ is also prime is
\[
\ll_{k} \frac{x}{W (\log x)^2} \Bigg(\prod_{p \mid W} \left(1 - \frac{1}{p}\right)^{-2}\Bigg) \Bigg(\prod_{\substack{p \mid h \\ p \nmid W}} \left(1 - \frac{1}{p}\right)^{-1}\Bigg),
\]
where the products are restricted to prime numbers $p$.
\end{lemma}

\noindent Whenever Lemma \ref{Th:selberg} is applied in Tao's proof of Theorem \ref{Tao}, the product over $p$ dividing $h$ is uniformly bounded. However, to prove Theorem \ref{thm:taoimproved}, we must deal with cases where that product can be very large. To work around this, we show that such cases arise very rarely, so rarely that this product is bounded in a suitable average sense. To establish this, we need to invoke a classical theorem of Romanoff \cite{romanoff34} about multiplicative orders, which originally appeared in his work on numbers of the form $p+2^k$. (Actually we use a slightly strengthened form of Romanoff's result due to Erd\H{o}s \cite{erdos51}.)

We would like to draw the interested reader's attention to the work reported on in \cite{FKNS10}, \cite{konyagin}, and \cite{GJRW14}, which also concerns problems connected with primality and digital expansions.


%
%

\subsection*{Notation}
We write $\Box \nmid n$ to indicate that $n$ is squarefree. The letter $p$ always denotes a prime. For a given integer $n$ we use $P(n)$ to denote the largest prime divisor of $n$ and $\omega(n)$ for the number of distinct prime factors of $n$. For a given integer $a$, we write $\ell_a(d)$ for the multiplicative order of $a$ modulo $d$. This notation reflects the importance in our analysis of considering $\ell_a(d)$ primarily as a function of $d$ rather than as a  function of $a$.

We use $f = O(g)$, or $f \ll g$, to mean that $|f| \le Cg$ for a suitable constant $C$. We use $f \gg g$ synonymously with $g \ll f$. If $f \ll g \ll f$, we write $f \asymp g$. We use $f = o(g)$ to mean $\lim f/g = 0$ as $N \to \infty$, holding other variables constant.

In what follows, \textbf{implied constants may depend on $K$}. Any further dependence (or independence) will be specified explicitly.
\section{Proof of Theorem \ref{thm:taoimproved}}


\subsection{A selective search}
We will confine our search for ``digitally delicate'' primes to primes lying in a certain conveniently chosen invertible residue class $b\bmod{W}$. (Cognoscenti will recognize this as an instance of the ``$W$-trick''.) To specify the residue class $b\bmod{W}$, we will require the use of a handful primes, determined by $K$ and an integer $M \ge K$.

\begin{lemma}
Let $K \ge 2$ be an integer and let $M \ge K$ also be an integer. There is a set $\mathcal{P}$ that is the disjoint union of sets $\mathcal{P} = \bigcup_{a = 2}^K \mathcal{P}_a$, where for each $2 \le a \le K$, $\mathcal{P}_a$ is a finite set of primes such that:

\begin{enumerate}[\upshape (i)]
\item \label{P1}for all $p \in \mathcal{P}_a$, we have $q_p:=P(a^p-1) > K$,
\item \label{P2} the primes $q_p$ are distinct for distinct $p \in \mathcal{P}$,
\item  \label{P3} $\displaystyle \sum_{p \in \mathcal{P}_a} \frac{1}{p} \ge M$.
\end{enumerate}
\end{lemma}

\begin{proof} According to a theorem of Stewart \cite[Theorem 1]{stewart77}, $P(a^p - 1) \gg p \log{p}$ for all primes $p$ and all $2 \le a \le K$. (See \cite{stewart13} for a more recent, much stronger estimate.) Keeping this mind, we construct the sets $\Pp_{a}$ inductively.  Given an integer $a$ with $2\le a \le K$, assume that the sets $\mathcal{P}_n$ have been constructed for all integers $2 \le n < a$.

We construct $\Pp_{a}$ as follows. By Stewart's result, we can pick $p_0$ so that whenever $p > p_0$, we have $P(a^p-1)$ larger than $K$ and larger than any element of $q_{p'}$ for $p' \in \bigcup_{2 \le n < a}\Pp_{n}$. As $p$ runs through the consecutive primes succeeding $p_0$, the numbers $P(a^p-1)$ are distinct, since the order of $a$ modulo $P(a^p-1)$ is precisely $p$. So we can construct $\mathcal{P}_a$ as the set of the first several consecutive primes exceeding $p_0$. Here ``first several'' means that we continue adding primes to $\Pp_{a}$ until (\ref{P3}) holds. This is possible due to the divergence of $\sum_p 1/p$ when $p$ is taken over all primes.
\end{proof}

We now set
\[ W = \prod_{p \in \mathcal{P}} q_p. \]
Observe that from Stewart's theorem quoted above,
\begin{equation}
\label{qo1}
\sum_{p \mid W} \frac{1}{p} = \sum_{p \in \mathcal{P}} \frac{1}{q_p} \ll \sum_{p} \frac{1}{p \log{p}} = O(1);
\end{equation}
here the final estimate follows, for example, by partial summation along with the prime number theorem.

Assume $M$ is sufficiently large in terms of $K$. Then we can partition each $\mathcal{P}_a$ into disjoint sets $\mathcal{P}_{a,j,k,s}$ such that
\[
\mathcal{P}_a = \bigcup_{1 \le |j| \le K} \bigcup_{k = 1}^{K} \bigcup_{s \in \Ss_N} \mathcal{P}_{a, j, k, s}
\]
and for each $\mathcal{P}_{a, j, k, s}$ we have
\begin{equation}\label{eq:moderatelymany}
\sum_{p \in \mathcal{P}_{a, j, k, s}} \frac{1}{p} \gg M.
\end{equation}

We now make our choice of residue class $b\bmod{W}$. Suppose $2\le a\le K$, $1\le |j|, k \le K$, and $s \in \Ss_N$. Let $p \in \Pp_{a,j,k,s}$. Since $q_p > K \ge k$, we know that $k^{-1}$ exists modulo $q_p$. Moreover, at least one of the two residue classes $k^{-1}(j+s) \bmod{q_p}$ or $k^{-1}(ja+s) \bmod{q_p}$ is invertible. Pick one, and say it is $b_p\bmod{W}$. We determine $b\bmod{W}$ as the solution to the simultaneous congruences
\[ b \equiv -b_p \pmod{q_p} \quad\text{for all $p \in \Pp$}. \]
Note that $b\bmod{W}$ is indeed a coprime residue class.

\subsection{Some initial reductions}
In what follows, we will always assume $N$ is sufficiently large in terms of fixed parameters $M$ and $K$. (Ultimately, $M$ will be chosen sufficiently large in terms of $K$.) Let
\[
Q_N := \#\{m \in {[N, (1 + \frac{1}{K})N]}: m \equiv b \pmod{W},\, m \text{ prime}\}.
\]
By the prime number theorem for arithmetic progressions, \[ Q_N \gg \frac{N}{\phi(W) \log{N}}.\] We would like to show that the same lower bound holds even after removing from our count those $m$ having $|km+ja^i+s|$ noncomposite (and $\ne p$) for some $1\le a,|j|,k\le K$, $0\le i \le K\log{N}$, and $s\in \Ss_N$.

We first dispense with those cases when $|km+ja^i+s|$ is noncomposite in virtue of having $|km+ja^i+s| \le 1$. Let
\begin{multline*}
E := \#\{p \in {[N, (1 + \frac{1}{K})N]}: m \equiv b \pmod{W},\,m \text{ prime},
\\
|km + ja^i + s| \le 1 \text{ for some value of } a, i, j, k, s\}
\end{multline*}
A given combination of $a, i, j, k$, and $s$ can contribute only $O(1)$ elements $m$ to $E$, so we have $E \ll \log{N}$. This bound is clearly $o(Q_N)$, and so is negligible for us.

It remains to discard those $m$ having $|km+ja^i+s|$ prime (and $\ne m$) for some $a,i,j,k,s$ as above. We may assume $ja^i+s\ne 0$. Otherwise $|km|$ is prime, forcing $k=1$ and $|km+ja^i+s|=m$, contrary to hypothesis.

The next easiest series of cases correspond to $a=1$. In these cases, $|km+j+s|$ is prime (and $\ne m$) for some $1\le |j|,k \le K$ and $s\in \Ss_N$. Given $j,k$, and $s$, the number of $m$ we must discard here is, by Lemma \ref{Th:selberg},
\[ \ll \frac{N}{W (\log{N})^2} \left(\prod_{p \mid W}\left(1-\frac{1}{p}\right)^{-2}\right) \left(\prod_{p\mid j+s}\left(1-\frac{1}{p}\right)^{-1}\right).\]
From \eqref{qo1}, the product on $p$ dividing $W$ is $O(1)$. Since $0 < |j+s| \le K(1+N)$, the product on $p$ dividing $j+s$ cannot exceed $O(\log\log{N})$; see \cite[Theorem 328, p. 352]{HW08}. Summing on the $O(1)$ possibilities for $j,k,s$, we see we must discard a total of
\[\ll \frac{1}{W} \frac{N \log\log{N}}{(\log{N})^2}\]
primes $m$ from these cases. This is $o(Q_N)$.

Naturally, the heart of the proof is the consideration of those cases when $a\ge 2$. Let
\begin{multline*}
{Q}_{N, a, i, j, k, s} := \#\{m \in {[N,(1 + \frac{1}{K})N]}: m \equiv b \pmod{W},
\\
m \text{ prime}, \text{ and } |km + ja^i + s| \text{ prime and $\ne m$}\}.
\end{multline*}
In the next section, we will show that
\begin{equation}\label{eq:ourgoal}
\sum_{a = 2}^K \sum_{0 \le i \le K \log{N}} \sum_{1 \le |j| \le K} \sum_{k = 1}^K \sum_{s \in \Ss_N} {Q}_{N, a, i, j, k, s} \ll \frac{N}{W\log{N}} \exp(-\frac{1}{2}\theta_K M)
\end{equation}
for a certain constant $\theta_K > 0$. Fixing $M$ sufficiently large in terms of $K$, we see that these values of $a$ force us to discard at most (say) $\frac{1}{2}Q_N$ primes.

Collecting the above estimates, we find that there are $\gg Q_N \gg N/\log{N}$ remaining primes $m$, all of which are digitally delicate in the strong sense of Theorem \ref{thm:taoimproved}.

%
%

\subsection{Detailed counting}
In this section, we establish the claimed upper bound on
\[
\sum_{a = 2}^K \sum_{0 \le i \le K\log{N}} \sum_{1 \le |j| \le K} \sum_{k = 1}^K \sum_{s \in \Ss_N} {Q}_{N, a, i, j, k, s}.
\]

We first handle the sum on $i$. For now treat $a, j, k$, and $s$ as fixed and consider
\begin{equation}
\label{i}
\sum_{0 \le i \le K\log{N}} Q_{N, a, i, j, k, s}.
\end{equation}
Because of our careful choice of $b$, either $kb + j + s \equiv 0 \pmod{q_p}$, or $kb + ja + s \equiv 0 \pmod{q_p}$ for all our $p \in \mathcal{P}_{a, j, k, s}$. In the former case, if $i \equiv 0 \pmod{p}$ for some $p \in \mathcal{P}_{a, j, k, s}$, then $q_p \mid{kb + ja^i + s}$. In the latter case, the same divisibility holds instead when $i \equiv 1 \pmod{p}$. (To see these results, recall that $a^p\equiv 1\pmod{q_p}$, by the choice of $q_p$.) If $q_p \mid{kb + ja^i + s}$ then at most two values of $m$ for a given $a, i, j, k$, and $s$ can have $|km + ja^i + s|$ prime: those where $|km + ja^i + s| = q_p$. So the number of $m$ contributed to (\ref{i}) in this way is $O(\log{N})$.

We thus focus on the remaining values of $i$. Let
\[
\mathcal{I} := \{0 \le i \le K \log{N}: \text{ for all } p \in \mathcal{P}_{a, j , k, s},~q_p \nmid kb + ja^i + s \},
\]
where $\mathcal{I}$ is understood to depend on the given $a, j, k$, and $s$. Then
\begin{equation}\label{eq:ibound} \#\mathcal{I} \ll \left(\prod_{p \in \mathcal{P}_{a, j, k, s}} \left(1 - \frac{1}{p}\right)\right) \log{N}. \end{equation} Moreover,
\begin{equation}\label{eq:iagain}
\sum_{0 \le i \le K\log{N}} Q_{N, a, i, j, k, s}
 \ll \log{N} + \sum_{i \in \mathcal{I}} Q_{N,a,i,j,k,s}.
\end{equation}
Whenever $ja^i+s= 0$, the quantity $Q_{N,a,i,j,k,s}$ vanishes, and so the final sum on $i$ can be restricted to those values with $ja^i+s\ne 0$.
By another application of Lemma \ref{Th:selberg}, as long as $ja^i+s\ne 0$,
\[
{Q}_{N, a, i, j, k, s} \ll \frac{N}{W (\log{N})^2} \prod_{p \mid ja^i + s} \left(1 - \frac{1}{p}\right)^{-1}.
\]
(We omitted the product over $p$ dividing $W$ here, since \eqref{qo1} shows that product is $\asymp 1$.)
By the Cauchy--Schwarz inequality,
\begin{equation}
\label{eq:CS}
\sum_{\substack{i \in \mathcal{I} \\ ja^i+s\ne 0}} \prod_{p \mid ja^i+s} \left(1-\frac{1}{p}\right)^{-1} \le \Bigg(\sum_{i \in \mathcal{I}} 1\Bigg)^{1/2} \Bigg(\sum_{\substack{0 \le i \le K\log{N} \\ ja^i+s\ne 0}} \prod_{p \mid ja^i+s} \left(1 - \frac{1}{p}\right)^{-2}\Bigg)^{1/2}.
\end{equation}
The first right-hand sum simply counts the number of $i \in \mathcal{I}$, and so from \eqref{eq:moderatelymany} and \eqref{eq:ibound},
\begin{equation}
\label{eq:CSfirst}
\sum_{i \in \mathcal{I}} 1 = \#\mathcal{I} \ll \left(\prod_{p \in \mathcal{P}_{a, j, k, s}} \left(1 - \frac{1}{p}\right)\right) \log{N} \ll \exp(-\theta_K M) \cdot\log{N},
\end{equation}
for a constant $\theta_K > 0$.
To estimate the second sum of (\ref{eq:CS}), we begin by observing that
$(1 - \frac{1}{p})^{-2} = (1 + \frac{2}{p})(1 + \frac{3p-2}{p^3-3p+2})$, and that
\[ \prod_{p} \left(1 + \frac{3p-2}{p^3-3p+2}\right) \le \exp\left(\sum_{p}\frac{3p-2}{p^3-3p+2} \right)<\infty, \]
where the products and sums are over all primes $p$. Thus,
\[
\prod_{p \mid ja^i + s} \left(1 - \frac{1}{p}\right)^{-2} \ll \prod_{p \mid ja^i + s} \left(1 + \frac{2}{p}\right).
\]
We claim that truncating the last product to primes $p\le \log{N}$ will not change its magnitude. To see this, observe that
\[
\prod_{\substack{p \mid ja^i + s \\ p > \log{N}}} \left(1 + \frac{2}{p}\right) \le \exp\bigg(\frac{2}{\log{N}} \sum_{\substack{p \mid ja^i + s \\ p > \log{N}}} 1\bigg) \le \exp\left(\frac{2}{\log{N}} \frac{\log{|ja^i +s|}}{\log{\log{N}}}\right).
\]
Put $Z:= K\cdot K^{K\log{N}} + KN$. Since $|ja^i + s| \le |j|a^i + |s| \le Z$, we have $\log{|ja^i + s|} \le \log{Z} \ll \log{N}$, and so final expression in the preceding display is $O(1)$. Consequently,
\[
\prod_{p \mid ja^i + s} \left(1 + \frac{2}{p}\right) \ll \prod_{\substack{p \mid ja^i + s \\ p\le \log{N}}} \left(1 + \frac{2}{p}\right),
\]
as claimed. Now rewrite
\[
\prod_{\substack{p \mid ja^i + s \\ p \le \log{N}}} \left(1 + \frac{2}{p}\right) = \sum_{\substack{d \mid ja^i + s \\ p \mid d \Rightarrow p \le \log{N} \\ \Box \nmid d}} \frac{2^{\omega(d)}}{d}.
\]
Assembling the above, we can estimate the second sum in \eqref{eq:CS} as follows:
\begin{align*}
\sum_{\substack{0 \le i \le K\log{N} \\ ja^i+s\ne 0}} \prod_{p \mid ja^i + s} \left(1 - \frac{1}{p}\right)^{-2} &\ll \sum_{\substack{0 \le i \le K\log{N} \\ ja^i+s\ne 0}} \sum_{\substack{d \mid ja^i + s \\ p \mid d \Rightarrow p \le \log{N} \\ \Box \nmid d}} \frac{2^{\omega(d)}}{d}
\\
&\ll \sum_{\substack{d \le Z \\ p \mid d \Rightarrow p \le \log{N} \\ \Box \nmid d}} \frac{2^{\omega(d)}}{d} \sum_{\substack{0 \le i \le K \log{N} \\ d \mid ja^i + s}} 1.
\end{align*}
Suppose $i \ge 1$ is such that $d \mid ja^i + s$. Defining $B_d:=\gcd(d,ja)$, and keeping in mind that $d$ is squarefree, we find that \begin{equation}\label{eq:lotsofgcds} \gcd(d,s) = \gcd(d, ja^i) = \gcd(d,ja) = B_d, \end{equation}
and
\[\frac{ja}{B_d} \cdot a^{i-1} \equiv -\frac{s}{B_d} \pmod{\frac{d}{B_d}}.\] This congruence, along with \eqref{eq:lotsofgcds}, shows that $a^{i-1}$ belongs to a uniquely determined coprime residue class modulo $d/B_d$. Thus, $i$ belongs to a fixed residue class modulo $\ell_a(d/B_d)$, and so
\begin{align}
\sum_{\substack{d \le Z \\ p \mid d \Rightarrow p \le \log{N} \\ \Box \nmid d}} \frac{2^{\omega(d)}}{d} \sum_{\substack{0 \le i \le K \log{N} \\ ja^i \equiv -s \pmod{d}}} 1 \ll \sum_{\substack{d \le Z \\ p \mid d \Rightarrow p \le \log{N} \\ \Box \nmid d}} \frac{2^{\omega(d)}}{d} \left(\frac{K \log{N}}{\ell_a\left(\frac{d}{B_d}\right)} + 1\right)\notag
\\
\label{sumd}
\ll \log{N} \sum_{\substack{d \in \mathbf{N} \\ (d/B_d,a) = 1}} \frac{2^{\omega(d)}}{d \cdot \ell_a\left(\frac{d}{B_d}\right)} + \sum_{\substack{d \le Z \\ p \mid d \Rightarrow p \le \log{N} \\ \Box \nmid d}} \frac{2^{\omega(d)}}{d}.
\end{align}
To handle the first right-hand sum, write $d = B_d d'$. Since $B_d \mid ja$,
\[
\sum_{\substack{d \in \mathbf{N} \\ (a, d/B_d) = 1}} \frac{2^{\omega(d)}}{d \cdot \ell_a\left(\frac{d}{B_d}\right)}
\le \sum_{B \mid ja} \sum_{\substack{d' \in \mathbf{N} \\ (a, d') = 1}} \frac{2^{\omega(Bd')}}{Bd' \cdot \ell_a(d')}
 \le \sum_{B \mid ja} \frac{2^{\omega(B)}}{B} \sum_{\substack{d' \in \mathbf{N} \\ (a, d') = 1}} \frac{2^{\omega(d')}}{d' \cdot \ell_a(d')},
\]
due to the complete subadditivity of $\omega$. Erd\H{o}s has proven a strengthening of Romanoff's theorem \cite[see Lemma 2, p. 417]{erdos51} saying that for any two positive integers $A$ and $S$, the series
\[
\sum_{\substack{n \in \mathbf{N} \\ (n, A) = 1}} \frac{S^{\omega(n)}}{n \cdot \ell_A(n)}
\]
is convergent. Taking $n=d'$, $A=a$, and $S = 2$, and noting that $a$, $j$, and $B$ are all $O(1)$, we see that the first of the two summands in (\ref{sumd}) is $O(\log{N})$.

To deal with the second summand of (\ref{sumd}), we reverse a  previous step and rewrite
\[
\sum_{\substack{d \le Z \\ p \mid d \Rightarrow p \le \log{N} \\ \Box \nmid d}} \frac{2^{\omega(d)}}{d}
= \prod_{p \le \log{N}} \left(1 + \frac{2}{p}\right)
\le \left(\prod_{p \le \log{N}} \left(1 - \frac{1}{p}\right)^{-1}\right)^2.
\]
By Mertens' Theorem (see \cite[Theorem 429, p. 466]{HW08}), the final expression is $O((\log{\log{N}})^2)$, which is certainly $O(\log{N})$. Hence,
\[
\sum_{\substack{0 \le i \le K\log{N} \\ ja^i+s\ne 0}} \prod_{p \mid ja^i + s} \left(1 - \frac{1}{p}\right)^{-2} \ll \log{N}.
\]
Substituting this estimate and (\ref{eq:CSfirst}) into (\ref{eq:CS}),
\[
\sum_{\substack{i \in \mathcal{I} \\ ja^i+s\ne 0}} \prod_{p \mid ja^i + s} \left(1 - \frac{1}{p}\right)^{-1}
\ll \log{N} \cdot \exp{\left(- \frac{1}{2} \theta_K M\right)}.
\]
We now deduce from \eqref{eq:iagain} that
\[
\sum_{0 \le i \le K \log{N}} Q_{N, a, i, j, k, s} \ll \frac{N}{W \log{N}} \exp{\left(- \frac{1}{2} \theta_K M\right)}.
\]
Finally, summing over the $O(1)$ possibilities for $a, j, k$, and $s$ yields \eqref{eq:ourgoal} and so completes the proof.

%
%




\section*{Acknowledgments} We would like to thank UGA's Center for Undergraduate Research Opportunities (CURO) for the opportunity to work together. Work of the first author is supported by the 2015 CURO Summer Fellowship, and work of the second author is supported by NSF award DMS-1402268. We are grateful to Christian Elsholtz for insightful comments.

\providecommand{\bysame}{\leavevmode\hbox to3em{\hrulefill}\thinspace}
\providecommand{\MR}{\relax\ifhmode\unskip\space\fi MR }
\providecommand{\MRhref}[2]{%
  \href{http://www.ams.org/mathscinet-getitem?mr=#1}{#2}
}
\providecommand{\href}[2]{#2}

\end{document}